\documentclass[10pt]{article}
 \font\smallit=cmti10

\usepackage{amssymb,latexsym,amsmath,epsfig,amsthm,color,algorithm,algorithmic} 

\makeatletter

\renewcommand{\@seccntformat}[1]{\csname the#1\endcsname. }

\makeatother

 \newtheorem{theorem}{Theorem}[section]
 \newtheorem{lemma}[theorem]{Lemma}
 
 \newtheorem{proposition}[theorem]{Proposition}
 
 \newtheorem{definition}[theorem]{Definition}

\begin{document}
\begin{center}
 {\bf Notes on the LVP and CVP in $p$-adic Fields}
 \vskip 30pt

 {\bf Chi Zhang and Mingqian Yao}\\
 {\smallit State Key Laboratory of Mathematical Sciences, Academy of Mathematics and Systems Science, Chinese Academy of Sciences, Beijing 100190, People's Republic of China}\\
 {and}\\
 {\smallit School of Mathematical Sciences, University of Chinese Academy of Sciences, Beijing 100049, People's Republic of China}\\

 \vskip 10pt

 {\tt zhangchi171@mails.ucas.ac.cn, yaomingqian@amss.ac.cn}\\

 \end{center}
 \vskip 30pt

\centerline{\bf Abstract}

This paper explores computational methods for solving the Longest Vector Problem (LVP) and Closest Vector Problem (CVP) in $p$-adic fields. Leveraging the non-Archimedean property of $p$-adic norms, we propose a polynomial time algorithm to compute orthogonal bases for $p$-adic lattices when the $p$-adic field is given by a minimal polynomial. The method utilizes the structure of maximal orders and $p$-radicals in extension fields of $\mathbb{Q}_{p}$ to efficiently construct uniformizers and residue field bases, enabling rapid solutions for the LVP and CVP. In addition, we introduce the characterization of norms on vector spaces over $\mathbb{Q}_p$.\par

\noindent

\pagestyle{myheadings}

 \thispagestyle{empty}
 \baselineskip=12.875pt
 \vskip 20pt

\section{Introduction}
The study of computational problems in non-Archimedean geometry, particularly over $p$-adic fields, has emerged as a significant area of intersection between algebraic number theory and theoretical cryptography. The Shortest Vector Problem (SVP) and the Closest Vector Problem (CVP), central to lattice-based hardness assumptions in Euclidean spaces, extend naturally to $p$-adic vector spaces, where the non-Archimedean norm structure $|\cdot|_p$ (satisfying $|v + w|_p \leq  \max\{|v|_p, |w|_p\}$) enables qualitatively distinct geometric decompositions. These generalizations to $p$-adic vector spaces are the Longest Vector Problem (LVP) and the Closest Vector Problem (CVP). They are first introduced by Deng et al. \cite{ref-2} in 2018. These new problems are considered to be hard and potentially applicable for constructing public-key encryption cryptosystems and signature schemes. In 2021, by introducing a trapdoor function with an orthogonal basis of a $p$-adic lattice, Deng et al. \cite{ref-3.5} constructed the first signature scheme and public-key encryption cryptosystem based on $p$-adic lattices.\par

The experimental results demonstrated that the new schemes achieve good efficiency. As for security, they did not give formal proofs for these schemes since the problems in local fields are relatively very new. Recently, much research has been done to analyze security of these schemes. In 2025, Zhang \cite{ref-10} improved the LVP algorithm in local fields. The modified LVP algorithm is a deterministic polynomial time algorithm when the field is totally ramified and $p$ is a polynomial in the rank of the input lattice, hence can be utilized to break the schemes. This attack extremely reduced the security strength of the schemes. In order to avoid this attack, Zhang et al. \cite{ref-11} proposed a method to construct orthogonal bases in $p$-adic fields with a large residue degree, which is helpful to modify the $p$-adic signature scheme and public-key encryption cryptosystem.\par

Nevertheless, these previous efforts ignored an essential aspect that given a minimal polynomial of a local field, there are polynomial time algorithms to compute the maximal order of this field, such as the Round 2 Algorithm \cite{ref-1} and the Round 4 Algorithm \cite{ref-12}. Besides, unlike their real or complex counterparts, $p$-adic lattices admit orthogonal bases, see \cite{ref-9}. Orthogonal bases of $p$-adic lattices can be used to solve the LVP and CVP in polynomial time \cite{ref-3.5, ref-9}, and can be derived from maximal orders and uniformizers in $p$-adic fields. Combining together, they provide a framework to efficiently solve the LVP and CVP in $p$-adic fields.

Recent advancements in $p$-adic computational algebra have highlighted the role of $p$-radicals and residue field extensions in constructing orthogonal decompositions, yet algorithms for the general LVP or CVP instances remain underexplored. Existing deterministic methods for Euclidean lattices (e.g., LLL reduction) do not directly generalize due to the ultrametric inequality. We addresses this gap by presenting a polynomial time algorithm to compute orthogonal bases for $p$-adic lattices, thereby reducing the LVP and CVP to tractable subproblems over orthogonal subspaces.

Our contributions are :

\begin{enumerate}

\item Algorithmic Innovation: By integrating $p$-adic Gram-Schmidt orthogonalization with residue field computations, we design efficient solvers for the LVP and CVP in $p$-adic fields, proving their polynomial time complexity under standard assumptions.

\item Practical Validation: A concrete example over $\mathbb{Q}_{2}(\sqrt{5})$ demonstrates the algorithm’s feasibility in computing orthogonal bases and solving lattice problems.

\end{enumerate}

The paper is structured as follows: Section 2 reviews $p$-adic norms, local fields, and lattice definitions. Section 3 introduces maximal orders and $p$-radicals as tools for enlarging subrings. Section 4 proves the existence of orthogonal bases via uniformizers and residue field extensions. Section 5 details the LVP and CVP algorithms, while Section 6 provides an example. Section 7 introduces the characterization of norms on vector spaces over $\mathbb{Q}_p$.

By unifying algebraic techniques with computational efficiency, this work advances the theoretical understanding of $p$-adic lattice problems, with implications for post-quantum cryptography and the arithmetic of $p$-adic extensions.

\section{Preliminaries}\label{se-2}

\subsection{Norm and Orthogonal Basis}

Let $p$ be a prime. Let $V$ be a vector space over $\mathbb{Q}_p$. A norm on $V$ is a function
$$|\cdot|:V\rightarrow\mathbb{R}$$
such that

\begin{enumerate}
\item $\left|v\right|\ge0$ for any $v\in V$, and $\left|v\right|=0$ if and only if $v=0$;
\item $\left|xv\right|=\left|x\right|_{p}\cdot \left|v\right|$ for any $x\in\mathbb{Q}_p$ and $v\in V$;
\item $\left|v+w\right|\le\max{\left\{\left|v\right|,\left|w\right|\right\}}$ for any $v,w\in V$.
\end{enumerate}

If $|\cdot|$ is a norm on $V$, and if $\left|v\right|\ne \left|w\right|$ for $v,w\in V$, then we must have $\left|v+w\right|=\max{\{\left|v\right|,\left|w\right|\}}$. Weil (\cite{ref-6} page 26) proved the following proposition.

\begin{proposition}[\cite{ref-6}]\label{pr-2.1}
Let $V$ be a vector space over $\mathbb{Q}_p$ of finite dimension $n>0$, and let $\left|\cdot\right|$ be a norm on $V$. Then there is a decomposition $V=V_1+V_2+\cdots+V_n$ of $V$ into a direct sum of subspaces $V_i$ of dimension $1$, such that
$$\left|\sum_{i=1}^{n}{v_i}\right|=\max_{1\le i\le n}{\left|v_i\right|}$$
for any $v_i\in V_i$, $i=1,2,\dots,n$.
\end{proposition}

Weil proved the above proposition for finite-dimensional vector spaces over a $p$-field (commutative or not). For simplicity, we only consider the case $\mathbb{Q}_p$. Thus, we can define the orthogonal basis.

\begin{definition}[orthogonal basis \cite{ref-6}]
Let $V$ be a vector space over $\mathbb{Q}_p$ of finite dimension $n>0$, and let $|\cdot|$ be a norm on $V$. We call $\alpha_1,\alpha_2,\dots,\alpha_n$ an orthogonal basis of $V$ over $\mathbb{Q}_p$ if $V$ can be decomposed into the direct sum of $n$ $1$-dimensional subspaces $V_i={\rm Span}_{\mathbb{Q}_p}(\alpha_i)$ $(1\le i\le n)$, such that
$$\left|\sum_{i=1}^{n}{v_i}\right|=\max_{1\le i\le n}{\left|v_i\right|}$$
for any $v_i\in V_i$, $i=1,2,\dots,n$.
Two subspaces $U$, $W$ of $V$ are said to be orthogonal if the sum $U+W$ is a direct sum and it holds that $\left|u+w\right|=\max\left\{\left|u\right|, \left|w\right|\right\}$ for all $u\in U$, $w\in W$.
\end{definition}

\subsection{Basic Facts about Local Fields}
In this subsection, we recall some basic facts about local fields. More details about local fields can be found in \cite{ref-4,ref-5}.\par
Let $p$ be a prime. For $x\in\mathbb{Q}$ with $x\ne0$, write $x=p^{t}\cdot\frac{a}{b}$ with $t,a,b\in\mathbb{Z}$ and $p\nmid ab$. Define $\left|x\right|_p=p^{-t}$ and $\left|0\right|_p=0$. Then $|\cdot|_p$ is a non-Archimedean absolute value on $\mathbb{Q}$. Namely, we have:

\begin{enumerate}
\item $\left|x\right|_p\ge0$ and $\left|x\right|_p=0$ if and only if $x=0$;
\item $\left|xy\right|_p=\left|x\right|_p\left|y\right|_p$;
\item $\left|x+y\right|_p\le\max{\left\{\left|x\right|_p,\left|y\right|_p\right\}}$. If $\left|x\right|_p\ne\left|y\right|_p$, then $\left|x+y\right|_p=\max{\left\{\left|x\right|_p,\left|y\right|_p\right\}}$.
\end{enumerate}

Let $\mathbb{Q}_p$ be the completion of $\mathbb{Q}$ with respect to $|\cdot|_p$. Denote
$$\mathbb{Z}_p=\left\{x\in\mathbb{Q}_p:\left|x\right|_p\le1\right\}.$$
We have
$$\mathbb{Q}_p=\left\{\sum_{i=j}^{\infty}a_ip^i:a_i\in\left\{0,1,2,\dots,p-1\right\},i\ge j,j\in\mathbb{Z}\right\},$$
and
$$\mathbb{Z}_p=\left\{\sum_{i=0}^{\infty}a_ip^i:a_i\in\left\{0,1,2,\dots,p-1\right\},i\ge 0\right\}.$$
$\mathbb{Z}_p$ is compact and $\mathbb{Q}_p$ is locally compact. $\mathbb{Z}_p$ is a discrete valuation ring, it has a unique nonzero principal maximal ideal $p\mathbb{Z}_p$ and $p$ is called a uniformizer of $\mathbb{Q}_p$. The unit group of $\mathbb{Z}_p$ is
$$\mathbb{Z}_p^{\times}= \left\{x\in\mathbb{Q}_p:\left|x\right|_p=1\right\}.$$
The residue class field $\mathbb{Z}_p/p\mathbb{Z}_p$ is a finite field with $p$ elements.\par
Let $n$ be a positive integer, and let $K$ be an extension field of $\mathbb{Q}_p$ of degree $n$. We fix some algebraic closure $\overline{\mathbb{Q}}_p$ of $\mathbb{Q}_p$ and view $K$ as a subfield of $\overline{\mathbb{Q}}_p$. Such $K$ exists. For example, let $K=\mathbb{Q}_p(\alpha)$ with $\alpha^n=p$. Because $X^n-p$ is an Eisenstein polynomial over $\mathbb{Q}_p$, it is irreducible over $\mathbb{Q}_p$, then $K$ has degree $n$ over $\mathbb{Q}_p$. The $p$-adic absolute value (or norm) $|\cdot|_p$ on $\mathbb{Q}_p$ can be extended uniquely to $K$, i.e., for $x\in K$, we have
$$\left|x\right|_p=\left| N_{K/\mathbb{Q}_p}(x)\right|_p^{\frac{1}{n}},$$
where $N_{K/\mathbb{Q}_p}$ is the norm map from $K$ to $\mathbb{Q}_p$. Therefore, $|\cdot|_p$ will also denote the unique absolute value (or norm) on $K$ extended by the $p$-adic absolute value on $\mathbb{Q}_p$.\par
Denote
$$\mathcal{O}_K=\left\{x\in K:\left|x\right|_p\le1\right\}.$$
$\mathcal{O}_K$ is also a discrete valuation ring, it has a unique nonzero principal maximal ideal $\pi\mathcal{O}_K$ and $\pi$ is called a uniformizer of $K$. $\mathcal{O}_K$ is a free $\mathbb{Z}_p$-module of rank $n$. $\mathcal{O}_K$ is compact and $K$ is locally compact. The unit group of $\mathcal{O}_K$ is
$$\mathcal{O}_K^{\times}= \left\{x\in K:\left|x\right|_p=1\right\}.$$
The residue field $\mathcal{O}_K/\pi\mathcal{O}_K$ is a finite extension of $\mathbb{Z}_p/p\mathbb{Z}_p$. Call the positive integer
$$f=\left[\mathcal{O}_K/\pi\mathcal{O}_K:\mathbb{Z}_p/p\mathbb{Z}_p\right]$$
the residue field degree of $K/\mathbb{Q}_p$. As ideals in $\mathcal{O}_K$, we have $p\mathcal{O}_K=\pi ^e\mathcal{O}_K$. Call the positive integer $e$ the ramification index of $K/\mathbb{Q}_p$. We have $n=\left[K:\mathbb{Q}_p\right]=ef$. When $e=1$, the extension $K/\mathbb{Q}_p$ is unramified, and when $e=n$, $K/\mathbb{Q}_p$ is totally ramified. Each element $x$ of the multiplicative group $K^{\times}$ of nonzero elements of $K$ can be written uniquely as $x=u\pi^t$ with $u\in\mathcal{O}_K^{\times}$ and $t\in\mathbb{Z}$. We have $p=u\pi^e$ with $u\in\mathcal{O}_K^{\times}$, so $\left|\pi\right|_p=p^{-\frac{1}{e}}$.

\subsection{$p$-adic Lattice}

We recall the definition of a $p$-adic lattice.

\begin{definition}[$p$-adic lattice \cite{ref-3}]
Let $V$ be a vector space over $\mathbb{Q}_p$ of finite dimension $n>0$. Let $m$ be a positive integer with $1\le m\le n$. Let $\alpha_1,\alpha_2,\dots,\alpha_m\in V$ be $\mathbb{Q}_p$-linearly independent vectors. A $p$-adic lattice in $V$ is the set
$$\mathcal{L}(\alpha_1,\alpha_2,\dots,\alpha_m):=\left\{\sum^{m}_{i=1}{a_i\alpha_i}:a_i\in\mathbb{Z}_p,1\le i\le m\right\}$$
of all $\mathbb{Z}_p$-linear combinations of $\alpha_1,\alpha_2,\dots,\alpha_m$. The sequence of vectors $\alpha_1,\alpha_2,\dots,\alpha_m$ is called a basis of the lattice $\mathcal{L}(\alpha_1,\alpha_2,\dots,\alpha_m)$. The integers $m$ and $n$ are called the rank and dimension of the lattice, respectively. When $n=m$, we say that the lattice is of full rank.
\end{definition}

We can also define the orthogonal basis of a $p$-adic lattice.

\begin{definition}[orthogonal basis of a $p$-adic lattice \cite{ref-3}]
Let $V$ be nontrivial finite-dimensional vector space over $\mathbb{Q}_p$, and let $\alpha_1,\alpha_2,\dots,\alpha_n\in V$. If $\alpha_1,\alpha_2,\dots,\alpha_n\in V$ form an orthogonal basis of the vector space ${\rm Span_{\mathbb{Q}_p}}(\alpha_1,\alpha_2,\dots,\alpha_n)$, then we say $\alpha_1,\alpha_2,\dots,\alpha_n$ is an orthogonal basis of the lattice $\mathcal{L}(\alpha_1,\alpha_2,\dots,\alpha_n)$.
\end{definition}

\subsection{LVP and CVP}

Deng et al. \cite{ref-2} introduced two new computational problems in $p$-adic lattices. They are the Longest Vector Problem (LVP) and the Closest Vector Problem (CVP). We review them briefly.

\begin{definition}[\cite{ref-2}]
Let $\mathcal{L}=\mathcal{L}(\alpha_1,\alpha_2,\dots,\alpha_m)$ be a lattice in $V$. We define recursively a sequence of positive real numbers $\lambda_1,\lambda_2,\lambda_3,\dots$ as follows.
$$\lambda_1=\max_{1\le i\le m}{\left|\alpha_i\right|_p},$$
$$\lambda_{j+1}=\max{\left\{\left|v\right|_p:v\in\mathcal{L},\left|v\right|_p<\lambda_j\right\}} \mbox{ for } j\ge1.$$
\end{definition}

We have $\lambda_1>\lambda_2>\lambda_3>\dots$ and $\lim_{j\rightarrow\infty}\lambda_j=0$. The Longest Vector Problem is defined as follows.

\begin{definition}[\cite{ref-2}]
Given a lattice $\mathcal{L}=\mathcal{L}(\alpha_1,\alpha_2,\dots,\alpha_m)$ in $V$, the Longest Vector Problem is to find a lattice vector $v\in\mathcal{L}$ such that $\left|v\right|_p=\lambda_2$.
\end{definition}

The Closest Vector Problem is defined as follows.

\begin{definition}[\cite{ref-2}]
Let $\mathcal{L}=\mathcal{L}(\alpha_1,\alpha_2,\dots,\alpha_m)$ be a lattice in $V$ and let $t\in V$ be a target vector. The Closest Vector Problem is to find a lattice vector $v\in\mathcal{L}$ such that $\left|t-v\right|_p=\min_{w\in\mathcal{L}}{\left|t-w\right|_p}$.
\end{definition}

Deng et al. \cite{ref-2} provided deterministic exponential time algorithms to solve the LVP and CVP. Additionally, Zhang et al. \cite{ref-9} presented deterministic polynomial time algorithms for solving the LVP and CVP specifically with the help of orthogonal bases.

\section{Computing the Maximal Order}\label{se-3}

Let $K=\mathbb{Q}_p(\theta)$ be a extension field over $\mathbb{Q}_p$ of degree $n$. We begin with computing the maximal order of $K$.

\begin{definition}[order \cite{ref-1}]
An order $\mathcal{O}$ in $K$ is a subring of $K$ which as a $\mathbb{Z}_p$-module is finitely generated and of maximal rank $n=\deg(K)$.
\end{definition}

The ring of algebraic integers $\mathcal{O}_K$ is called the maximal order. We will build it up by starting from a known order (in fact from $\mathbb{Z}_p[\theta]$) and by successively enlarging it. 

\begin{definition}[$p$-radical \cite{ref-1}]
Let $\mathcal{O}$ be an order in a $p$-adic field $K$. We define the $p$-radical $I_p$ as follows.
$$I_p=\left\{x\in\mathcal{O}\ |\ \exists m\ge1,\ x^m\in p\mathcal{O}\right\}$$ 
\end{definition}

With the help of the $p$-radical, we can enlarging an order $\mathcal{O}$ in a $p$-adic field $K$, as described in the following theorem: 

\begin{theorem}[\cite{ref-1}]\label{th-3.3}
Let $\mathcal{O}$ be an order in a $p$-adic field $K$. Set
$$\mathcal{O}^{\prime}=\left\{x\in K\ |\ xI_p\subset I_p\right\}.$$
Then either $\mathcal{O}^{\prime}=\mathcal{O}$, in which case $\mathcal{O}$ is maximal, or $\mathcal{O}^{\prime}\varsupsetneq\mathcal{O}$ and $p|[\mathcal{O}^{\prime}:\mathcal{O}]|p^n$.
\end{theorem}

Let $T$ be the minimal polynomial of $\theta$. We can write ${\rm disc}(T)=dp^m$, where $\gcd(d,p)=1$. If $\mathcal{O}_K$ is the maximal order which we are looking for, then the index $[\mathcal{O}_K:\mathbb{Z}[\theta]]$ divides $p^m$. We compute $\mathcal{O}^{\prime}$ as described in Theorem \ref{th-3.3}. If $\mathcal{O}^{\prime}=\mathcal{O}$, then $\mathcal{O}$ is maximal and we are finished. Otherwise, replace $\mathcal{O}$ by $\mathcal{O}^{\prime}$, and use the method of Theorem \ref{th-3.3} again. It is clear that this algorithm is valid and will lead quite rapidly to the maximal order.\par
Now, we explain how to compute $I_p$. It is simpler to compute in $R=\mathcal{O}/p\mathcal{O}$. To compute the $p$-radical of $R$, we need the following lemma: 

\begin{lemma}[\cite{ref-1}]\label{le-3.4}
If $n=[K:\mathbb{Q}_p]$ and if $j\ge1$ is such that $p^j\ge n$, then the $p$-radical of $R$ is equal to the kernel of the map $x\mapsto x^{p^j}$, which is the $j$th power of the Frobenius homomorphism. 
\end{lemma}

Let $\omega_1,\dots,\omega_n$ be a basis of $\mathcal{O}$. Then $\overline{\omega}_1,\dots,\overline{\omega}_n$ is an $\mathbb{F}_p$ basis of $R$. For $k=1,\dots,n$, we compute $\overline{a}_{i,k}$ such that
$$\overline{\omega}_k^{p^j}=\sum_{i=1}^n{\overline{a}_{i,k}\overline{\omega}_i},$$
Hence, if $\overline{A}$ is the matrix of the $\overline{a}_{i,k}$, the $p$-radical is simply the kernel of this matrix. We can apply Algorithm 2.3.1 in \cite{ref-1} to obtain a basis of $\overline{I_p}$, the $p$-radical of $R$. Then $I_p$ is generated by pullbacks of a basis of $\overline{I_p}$ and $p\omega_1,\dots,p\omega_n$.\par
Now that we have $I_p$, we must compute $\mathcal{O}^{\prime}$. For this, we use the following lemma:

\begin{lemma}[\cite{ref-1}]\label{le-3.5}
With the notations of Theorem \ref{th-3.3}, if $U$ is the kernel of the map
$$\alpha\mapsto\left(\overline{\beta}\mapsto\overline{\alpha\beta}\right)$$
from $\mathcal{O}$ to ${\rm End}(I_p/pI_p)$, then  $\mathcal{O}^{\prime}=\frac{1}{p}U$.
\end{lemma}

Algorithm 2.3.1 in \cite{ref-1} makes sense only over a field, so we must first compute the kernel $\overline{U}$ of the map $\mathcal{O}/p\mathcal{O}$ into ${\rm End}(I_p/pI_p)$. The matrix of the map is an $n^2\times n$ matrix. If $\overline{\gamma}_i$ are a basis of $I_p/pI_p$, one computes
$$\omega_k\overline{\gamma}_i=\sum_{1\le j\le n}{a_{k,i,j}\overline{\gamma}_j},$$
and $k$ is the column number, while $(i,j)$ is the row index. If $\overline{v}_1,\dots,\overline{v}_k$ is the basis of kernel $\overline{U}$, then $U$ is generated by $v_1,\dots,v_k,p\omega_1,\dots,p\omega_n$.

\section{Finding an Orthogonal Basis}\label{se-4}

Now that we obtain a basis of the maximal order $\mathcal{O}_K$. We can find a uniformizer $\pi$ from the basis of the $p$-radical of $\mathcal{O}_K$. It is the element of the maximal $p$-adic absolute value. Then we compute a basis $\overline{s_1},\dots,\overline{s_f}$ of the residue field $k=\mathcal{O}_K/\pi\mathcal{O}_K$ (details are given in Theorem \ref{th-5.1}). Finally, let $e=\frac{n}{f}$. We conclude that $(s_{i}\pi^{j})_{1\le i\le f,\ 0\le j\le e-1}$ is an orthogonal basis of $K$ according to the following results.

\begin{lemma}[\cite{ref-7} page 167, Exercise 5A]\label{le-4.1}
Let $K$ be an extension field of degree $n$ over $\mathbb{Q}_p$. Let $V$ be a subspace of $K$ over $\mathbb{Q}_p$. Assume that $\alpha_{1},\alpha_{2},\dots,\alpha_{m}$ $(m\le n)$ is a basis of $V$ over $\mathbb{Q}_p$ and $\left|\alpha_{1}\right|=\left|\alpha_{2}\right|=\cdots=\left|\alpha_{m}\right|$, which is denoted by $\lambda_{1}$. Let $\pi$ be a uniformizer of $K$, so there is an integer $s$ such that $\left|\pi^{s}\right|=\lambda_{1}$. Then $\alpha_{1},\alpha_{2},\dots,\alpha_{m}$ is an orthogonal basis of $V$ over $\mathbb{Q}_p$ if and only if $\overline{\alpha_{1}},\overline{\alpha_{2}},\dots,\overline{\alpha_{m}}$ are linearly independent over $\mathbb{F}_p$, where $\overline{\alpha_{i}}$ is the image of $\pi^{-s}\cdot\alpha_{i}$ in $k$.
\end{lemma}

A proof of Lemma \ref{le-4.1} can be found in \cite{ref-11}.

\begin{lemma}\label{le-4.2}
Let $K$ be an extension field over $\mathbb{Q}_p$. Let $V_{i}\subset K$ be a vector space over $\mathbb{Q}_p$ of finite dimension $n_{i}>0$, $1\le i\le s$. Let $\alpha_{i1},\alpha_{i2},\dots,\alpha_{in_{i}}$ be an orthogonal basis of $V_{i}$ over $\mathbb{Q}_p$. If
$$\left\{\left|v_{i}\right|\Big|v_{i}\in V_{i}\right\}\cap\left\{\left|v_{j}\right|\Big|v_{j}\in V_{j}\right\}=\{0\}$$
for all $1\le i<j\le s$. Then $\alpha_{11},\alpha_{12},\dots,\alpha_{1n_{1}},\dots,\alpha_{s1},\alpha_{s2},\dots,\alpha_{sn_{s}}$ is an orthogonal basis of $V=\bigoplus^{s}_{i=1}V_{i}$ over $\mathbb{Q}_p$.
\begin{proof}
We prove it by induction. If $s=1$, then the statement is obvious. If $s=2$, then by assumption, $V_{1}\cap V_{2}$=\{0\}. Hence $V=V_{1}\oplus V_{2}$ is a direct sum. For any $v\in V$, we have $v=v_{1}+v_{2}$ where $v_{i}=\sum^{n_i}_{j=1}a_{ij}\alpha_{ij}\in V_{i}$, $a_{ij}\in\mathbb{Q}_p$, $i=1,2$. Since
$$\left\{\left|v_{1}\right|\Big|v_{1i}\in V_{1}\right\}\cap\left\{\left|v_{2}\right|\Big|v_{2}\in V_{2}\right\}=\{0\},$$
we have $\left|v_{1}+v_{2}\right|=\max{\{\left|v_{1}\right|,\left|v_{2}\right|\}}$.
Then,
$$\left|v\right|=\left|v_{1}+v_{2}\right|=\max{\{\left|v_{1}\right|,\left|v_{2}\right|\}}=\max_{i,j}\left|a_{ij}\alpha_{ij}\right|.$$
Therefore,
$$\alpha_{11},\alpha_{12},\dots,\alpha_{1n_{1}},\alpha_{21},\alpha_{22},\dots,\alpha_{2n_{2}}$$
is an orthogonal basis of $V$ over $\mathbb{Q}_p$.\par
Suppose that it holds for $s=k\ge2$. For $s=k+1$, applying the induction assumption to $V_{1},V_{2},\dots,V_{k}$, we conclude that
$$\alpha_{11},\alpha_{12},\dots,\alpha_{1n_{1}},\dots,\alpha_{k1},\alpha_{k2},\dots,\alpha_{kn_{k}}$$
is an orthogonal basis of $V^{\prime}=\bigoplus^{k}_{i=1}V_{i}$ over $\mathbb{Q}_p$. Then, for any $v^{\prime}\in V^{\prime}$, we have 
$$v^{\prime}=\sum^{k}_{i=1}v_{i},$$
where $v_{i}\in V_{i}$ for $1\le i\le k$, and
$$\left|v^{\prime}\right|=\max_{1\le i\le k}{\left|v_{i}\right|}\in\bigcup^{k}_{i=1}\{\left|v_{i}\right||v_{i}\in V_{i}\}.$$
Therefore,
$$\left\{\left|v^{\prime}\right|\Big|v^{\prime}\in V^{\prime}\right\}\cap\left\{\left|v_{s}\right|\Big|v_{s}\in V_{s}\right\}=\{0\}.$$
Applying the induction assumption to $V^{\prime}$ and $V_{s}$, we conclude that
$$\alpha_{11},\alpha_{12},\dots,\alpha_{1n_{1}},\dots,\alpha_{s1},\alpha_{s2},\dots,\alpha_{sn_{s}}$$
is an orthogonal basis of $V=V^{\prime}\oplus V_{s}=\bigoplus^{s}_{i=1}V_{i}$ over $\mathbb{Q}_p$.
\end{proof}
\end{lemma}

\begin{theorem}\label{th-4.3}
Let $K$ be an extension field of degree $n$ over $\mathbb{Q}_p$. Its residue degree is $f$ and ramification index is $e$. Let $\pi$ be a uniformizer of $K$ and $(s_{i})_{1\le i\le f}$ be a family in $\mathcal{O}_K$ such that the image $\overline{s_{i}}\in k$ make up a basis of $k$ over $\mathbb{F}_p$. Then the family
$$(s_{i}\pi^{j})_{1\le i\le f,\ 0\le j\le e-1}$$
is an orthogonal basis of $K$ over $\mathbb{Q}_p$.
\begin{proof}
Since $n=ef$. We can prove that the elements in this family are linearly independent over $\mathbb{Q}_p$ (see  \cite{ref-4} page 99), so it is a basis of $K$ over $\mathbb{Q}_p$. Let $V_{j}$ be the vector space generated by $(s_i\pi^j)_{1\le i\le f}$ over $\mathbb{Q}_p$, $0\le j\le e-1$. Then $K=\bigoplus^{e-1}_{j=0}V_{j}$. Since $(\overline{s_{i}})_{1\le i\le f}$ are linearly independent over $\mathbb{F}_p$, by Lemma \ref{le-4.1}, $(s_i\pi^j)_{1\le i\le f}$ is an orthogonal basis of $V_{j}$ over $\mathbb{Q}_p$. Since $\left|\pi\right|=p^{-\frac{1}{e}}$, we have
$$\{\left|v_{j}\right||v_{j}\in V_{j}\}=\{0\}\cup p^{\mathbb{Z}-\frac{j}{e}}.$$
Then by Lemma \ref{le-4.2}, the family
$$(s_{i}\pi^{j})_{1\le i\le f,\ 0\le j\le e-1}$$
is an orthogonal basis of $K$ over $\mathbb{Q}_p$.
\end{proof}
\end{theorem}

\section{Solving the LVP and CVP with Orthogonal Bases}\label{se-5}

Once we obtain an orthogonal basis of $K$, we can use the following algorithms to solve the LVP and CVP in $K$.

\begin{algorithm}[H]
\caption{(LVP with orthogonal bases \cite{ref-9})}
\begin{algorithmic}[1]
\REQUIRE an orthogonal basis $e_1,\dots,e_n$ of $K$, a $p$-adic lattice $\mathcal{L}=\mathcal{L}(\alpha_1,\dots,\alpha_m)$ in $K$
\ENSURE a (second) longest vector of $\mathcal{L}$
\FOR {$i=1$ to $m$}
	\STATE rearrange $\alpha_i,\dots,\alpha_m$ such that  $\left|\alpha_i\right|_p=\max_{i\le k\le m}{\left|\alpha_k\right|_p}$
	\IF {$i>1$ and $\left|\alpha_{i-1}\right|_p>\left|\alpha_i\right|_p$}
		\STATE break
	\ENDIF
	\STATE rearrange $e_i,\dots,e_n$ such that  $\left|a_{ii}e_i\right|_p=\max_{i\le j\le m}{\left|a_{ij}e_j\right|_p}$
	\FOR {$l=i+1$ to $m$}
		\STATE $\alpha_l\leftarrow\alpha_l-\frac{a_{li}}{a_{ii}}\alpha_i$
	\ENDFOR
\ENDFOR
\IF {$\left|p\alpha_1\right|_p>\left|\alpha_i\right|_p$}
	\STATE $v\leftarrow p\alpha_1$
\ELSE
	\STATE $v\leftarrow\alpha_i$
\ENDIF
\RETURN $v$
\end{algorithmic}
\end{algorithm}

This algorithm runs in polynomial time in the input size if we can compute efficiently the norm $\left|v\right|_p$ of any vector $v\in V$.

\begin{algorithm}[H]
\caption{(CVP with orthogonal bases \cite{ref-9})}
\begin{algorithmic}[1]
\REQUIRE an orthogonal basis $e_1,\dots,e_n$ of $K$, a $p$-adic lattice $\mathcal{L}=\mathcal{L}(\alpha_1,\dots,\alpha_m)$ in $K$, a target vector $t\in K$
\ENSURE a closest lattice vector $v$ of $t$
\STATE $v\leftarrow0$, $s\leftarrow t$, write $s=\sum^{n}_{j=1}{s_je_j}$
\FOR {$i=1$ to $m$}
	\STATE rearrange $\alpha_i,\dots,\alpha_m$ such that  $\left|\alpha_i\right|_p=\max_{i\le k\le m}{\left|\alpha_k\right|_p}$
	\IF {$\left|s\right|_p>\left|\alpha_i\right|_p$}
		\STATE break
	\ENDIF
	\STATE rearrange $e_i,\dots,e_n$ such that  $\left|a_{ii}e_i\right|_p=\max_{i\le j\le m}{\left|a_{ij}e_j\right|_p}$
	\STATE $s\leftarrow s-\frac{s_i}{a_{ii}}\alpha_i$, $v\leftarrow v+\frac{s_i}{a_{ii}}\alpha_i$
	\IF {$s=0$}
		\STATE break
	\ENDIF
	\FOR {$l=i+1$ to $m$}
		\STATE $\alpha_l\leftarrow\alpha_l-\frac{a_{li}}{a_{ii}}\alpha_i$
	\ENDFOR
\ENDFOR
\RETURN $v$
\end{algorithmic}
\end{algorithm}

This algorithm also runs in polynomial time in the input size if we can compute efficiently the norm $\left|v\right|_p$ of any vector $v\in K$.\par
Finally, we summrize the progress of solving the LVP and CVP in $p$-adic fields as the following theorem.

\begin{theorem}\label{th-5.1}
Let $K$ be an extension field of degree $n$ over $\mathbb{Q}_p$. Then there exists a deterministic algorithm to solve the LVP and CVP of any $p$-adic lattice $\mathcal{L}$ in $K$. The time complexity of this algorithm is polynomial in $n$ and $\log p$.
\end{theorem}
\begin{proof}
Let $e$ be the ramfication index and $f$ be the residue degree. First, we use Theorem \ref{th-3.3} to compute the maximal order $\mathcal{O}_K$. We also use Lemma \ref{le-3.4} to compute a basis of $\left(\mathcal{O}_K\right)_p$, the $p$-radical of $\mathcal{O}_K$. Let $\pi$ be the element of the maximal $p$-adic absolute value in this basis. We conclude that $\pi$ is a uniformizer, i.e., $|\pi|_p=p^{-\frac{1}{e}}$. Otherwise, there is no vectors of $p$-adic absolute value $p^{-\frac{1}{e}}$ in $\left(\mathcal{O}_K\right)_p$. However, any uniformizer belongs to $\left(\mathcal{O}_K\right)_p$ and is of $p$-adic absolute value $p^{-\frac{1}{e}}$. This is a contradiction.\par
Then, since there exists a unique unramified extension of degree $f$ over $\mathbb{Q}_p$. According to \cite{ref-8}, to find  a polynomial generating this extension, we look at random monic polynomials of degree $f$ over $\mathbb{F}_p=\mathbb{Q}_p/p\mathbb{Q}_p$ until we find an irreducible one. Let $s_1,\dots,s_f$ be an integral basis of this extension. We conclude that $\overline{s_1},\dots,\overline{s_f}$ is a basis of the residue field $k=\mathcal{O}_K/\pi\mathcal{O}_K$. By Theorem \ref{th-4.3}, $(s_{i}\pi^{j})_{1\le i\le f,\ 0\le j\le e-1}$ is an orthogonal basis of $K$.\par
Finally, we can use algorithms in Section \ref{se-5} to solve the LVP and CVP of any $p$-adic lattice $\mathcal{L}$ in $K$.\par
Since the subroutines are all polynomial time algorithms in $n$ and $\log p$, this is a polynomial time algorithm in $n$ and $\log p$.
\end{proof}

\section{A Toy Example}\label{se-6}

Let $K=\mathbb{Q}_2(\theta)$ be a extension field over $\mathbb{Q}_2$ where $\theta$ is a root of $F(X)=X^3-3X^2+3X-5$. To solve the LVP and CVP in $K$, we need to find an orthogonal basis of $K$. We begin with $\mathcal{O}=\mathbb{Z}_2[\theta]$. Choose $1,\theta,\theta^2$ as a basis of $\mathcal{O}$.\par
First, Let us compute the $2$-radical $I_2$. Since $[K:\mathbb{Q}_2]=3$, we can take $j=2$. Then we compute the matrix $\overline{A}$ of the map $x\mapsto x^4$
$$\overline{A}=
\begin{pmatrix}
 1 & 1 & 1\\
 0 & 0 & 0\\
 0 & 0 & 0
\end{pmatrix}.$$
Hence, $\overline{\theta-1},\overline{\theta^2-1}$ is a basis of $\overline{I_2}$, and $2,\theta-1,\theta^2-1$ is a basis of $I_2$.\par
Next, Let us compute $\mathcal{O}^{\prime}$. Since $\overline{2},\overline{\theta-1},\overline{\theta^2-1}$ is a basis of $I_2/2I_2$, the matrix of the map
$$\alpha\mapsto\left(\overline{\beta}\mapsto\overline{\alpha\beta}\right)$$
is
$$\begin{pmatrix}
 1 & 1 & 1\\
 0 & 0 & 0\\
 0 & 0 & 0\\
 0 & 0 & 0\\
 1 & 1 & 1\\
 0 & 1 & 0\\
 0 & 0 & 0\\
 0 & 0 & 0\\
 1 & 1 & 1
\end{pmatrix}.$$
Hence, $\overline{\theta^2-1}$ is a basis of its kernel $\overline{U}$, and $2,2\theta,\theta^2-1$ is a basis of $U$. Therefore, $\mathcal{O}^{\prime}=\frac{1}{2}U$ is generated by $1,\theta,\frac{1}{2}(\theta^2-1)$.\par
In order to determine whether $\mathcal{O}^{\prime}$ is maximal, we need another round. Choose $1,\theta,\frac{1}{2}(\theta^2-1)$ as a basis of $\mathcal{O}^{\prime}$. The matrix $\overline{A^{\prime}}$ of the map $x\mapsto x^4$ is
$$\overline{A^{\prime}}=
\begin{pmatrix}
 1 & 1 & 0\\
 0 & 0 & 0\\
 0 & 0 & 0
\end{pmatrix}.$$
Hence, $\overline{\theta-1},\overline{\frac{1}{2}(\theta^2-1)}$ is a basis of $\overline{I^{\prime}_2}$, and $2,\theta-1,\frac{1}{2}(\theta^2-1)$ is a basis of $I_2^{\prime}$.\par
Since $\overline{2},\overline{\theta-1},\overline{\frac{1}{2}(\theta^2-1)}$ is a basis of $I_2^{\prime}/2I_2^{\prime}$, the matrix of the map $\alpha\mapsto\left(\overline{\beta}\mapsto\overline{\alpha\beta}\right)$ is
$$\begin{pmatrix}
 1 & 1 & 0\\
 0 & 0 & 0\\
 0 & 0 & 0\\
 0 & 0 & 1\\
 1 & 1 & 0\\
 0 & 0 & 0\\
 0 & 1 & 1\\
 0 & 0 & 1\\
 1 & 1 & 0
\end{pmatrix}.$$
Hence, its kernel $\overline{U^{\prime}}=\left\{\overline{0}\right\}$, and $2,2\theta,\theta^2-1$ is a basis of $U^{\prime}$. Therefore, $\mathcal{O}^{\prime\prime}=\frac{1}{2}U^{\prime}=\mathcal{O}^{\prime}$. We conclude that $\mathcal{O}^{\prime}$ is the maximal order.\par
Now, let us find an orthogonal basis of $K$. Since $\left|2\right|_p=2^{-1}$, $\left|\theta-1\right|_p=2^{-\frac{2}{3}}$ and $\left|\frac{1}{2}(\theta^2-1)\right|_p=2^{-\frac{1}{3}}$, we find a uniformizer $\pi=\frac{1}{2}(\theta^2-1)$ in $I_2^{\prime}$. The residue field $k=\mathcal{O}^{\prime}/\pi\mathcal{O}^{\prime}$ is generated by $1$. Therefore, $1,\pi,\pi^2$ is an orthogonal basis of $K$.\par
Once we obtain an orthogonal basis of $K$, we can use algorithms in section \ref{se-5} to solve the LVP and CVP in $K$. Examples of these algorithms are given in \cite{ref-9}.

\section{Characterization of Norms on Vector Spaces over $\mathbb{Q}_p$}

Since the norm given by the $p$-adic absolute value of a $p$-adic field is insecure, we need to consider other methods to give a computable norm. The following theorem characterizes norms on vector spaces over $\mathbb{Q}_p$. This theorem may be well known, though we could not find any references. For the sake of completeness, we provide a proof here.\par

\begin{theorem}\label{th-7.1}
Let $V$ be a vector space over $\mathbb{Q}_p$ of finite dimension $n>0$, which is given by basis $\alpha_1,\alpha_2,\dots,\alpha_n$. Then every norm $N$ on $V$ is of form
$$N\left(\sum_{i=1}^{n}{b_{i}\alpha_{i}}\right)=\max_{1\le j\le n}{c_j\left|\sum_{i=1}^{n}b_{i}a_{ij}\right|_p},$$
where $b_1,b_2,\dots,b_n\in\mathbb{Q}_p$, $c_1,c_2,\dots,c_n\in\mathbb{R}_{+}$ and $(a_{ij})_{n\times n}\in{\rm GL}_n(\mathbb{Q}_p)$. In other words, every norm on $V$ can be determined by $n$ positive real numbers and an $n\times n$ invertible matrix over $\mathbb{Q}_p$.
\end{theorem}
\begin{proof}
According to Proposition \ref{pr-2.1}, there is an orthogonal basis $e_1,e_2,\dots,e_n$ of $V$. Denote $N(e_i)=c_i\in\mathbb{R}$ for $1\le i\le n$ and
$$\begin{pmatrix}
 \alpha_1 \\
 \alpha_2 \\
 \vdots \\
 \alpha_n
\end{pmatrix}=A\cdot\begin{pmatrix}
 e_1 \\
 e_2 \\
 \vdots \\
 e_n
\end{pmatrix}$$
for some invertible matrix $A=(a_{ij})_{n\times n}\in{\rm GL}_n(\mathbb{Q}_p)$. For any vector $v=\sum_{i=1}^{n}b_i\alpha_i\in V$, we have
$$v=(b_1,b_2,\dots,b_n)\cdot\begin{pmatrix}
 \alpha_1 \\
 \alpha_2 \\
 \vdots \\
 \alpha_n
\end{pmatrix}=(b_1,b_2,\dots,b_n)\cdot A\cdot\begin{pmatrix}
 e_1 \\
 e_2 \\
 \vdots \\
 e_n
\end{pmatrix}.$$

Therefore,
$$N(v)=\max_{1\le j\le n}{N\left(\sum_{i=1}^{n}b_{i}a_{ij}e_j\right)}=\max_{1\le j\le n}{c_j\left|\sum_{i=1}^{n}b_{i}a_{ij}\right|_p}.$$
\end{proof}

If a norm is given as in Theorem \ref{th-7.1}, then it will be easy to compute. However, it is insecure, since we can find an orthogonal basis by the matrix $A$. Specifically, we have
$$\begin{pmatrix}
 e_1 \\
 e_2 \\
 \vdots \\
 e_n
\end{pmatrix}=A^{-1}\cdot\begin{pmatrix}
 \alpha_1 \\
 \alpha_2 \\
 \vdots \\
 \alpha_n
\end{pmatrix}.$$

\section{Conclusion}\label{se-7}

We propose a deterministic polynomial time algorithm to solve the LVP and CVP in $p$-adic fields. The previous work \cite{ref-10} attacked schemes under specific conditions, while the algorithm in this paper is a general attack for the case where the field is fully specified. Therefore, the public-key cryptosystems or signature schemes in \cite{ref-3.5} are insecure. Although the method in \cite{ref-11} for constructing orthogonal bases with large residue degree aimed to patch schemes, the current work shows the fundamental vulnerability remains if the field is known.\par
However, there is still room for improvement. For instance, instead of giving the whole $p$-adic field by a minimal polynomial, we could only give a way to compute the $p$-adic absolute value. In other words, we could consider public-key cryptosystems and signature schemes in vector spaces over $\mathbb{Q}_p$, since there are no known deterministic polynomial time algorithms to solve the LVP or CVP in vector spaces over $\mathbb{Q}_p$ when the norm is given as an oracle. The crucial point is to give a computable norm, while do not leak any orthogonal bases. On the other hand, we could ask that given an oracle of a norm, are there any polynomial time algorithms to find $c_1,c_2,\dots,c_n$ and $A$ in Theorem \ref{th-7.1}.

\section*{Acknowledgements}

This work was supported by National Natural Science Foundation of China(No. 12271517) and National Key R\&D Program of China(No. 2025YFA1017203). The idea of computing the maximal order comes from an anonymous referee. We extend our sincere appreciation to the referee.

%


\begin{thebibliography}{99}

\bibitem{ref-1} Henri Cohen, \textit{A Course in Computational Algebraic Number Theory}, Springer-Verlag, New York, 1993.
\bibitem{ref-12} D. Ford. On the computation of the maximal order in a dedekind domain. Ph.D thesis, Ohio State University, Ohio, 1978.
\bibitem{ref-2} Yingpu Deng, Lixia Luo, Yanbin Pan and Guanju Xiao, \textit{On Some Computational Problems in Local Fields}, Journal of Systems Science and Complexity, \textbf{35}, 1191–1200, 2022.
\bibitem{ref-3} Yingpu Deng, \textit{On $p$-adic Gram-Schmidt Orthogonalization Process}, Frontiers of Mathematics, \textbf{20}(2), 299-311, 2025.
\bibitem{ref-3.5} Yingpu Deng, Lixia Luo, Yanbin Pan, Zhaonan Wang and Guanju Xiao, \textit{Public-key Cryptosystems and Signature Schemes from p-adic Lattices}, $p$-Adic Numbers, Ultrametric Analysis and Applications, \textbf{16}, 23–42, 2024.
\bibitem{ref-4} A.M. Robert, \textit{A course in p-adic analysis}, GTM 198, Springer, New York, 2000.
\bibitem{ref-5} J.W.S. Cassels, \textit{Local fields}, Cambridge University Press, Cambridge, 1986.
\bibitem{ref-6} A. Weil, \textit{Basic number theory}, Third edition, Springer, New York, 1974.
\bibitem{ref-7} A.C.M. van Rooij, \textit{Non-Archimedean functional analysis}, Marcel Dekker, New York, 1978.
\bibitem{ref-8} S. Pauli and X. Roblot, \textit{On the computation of all extensions of a $p$-adic field of a given degree}, Math. Comput, \textbf{70}, 1641–1659, 2001.
\bibitem{ref-9} Chi Zhang, Yingpu Deng and Zhaonan Wang, \textit{Norm Orthogonal Bases and Invariants of $p$-adic Lattices}, Linear Algebra and its Applications, \textbf{728}, 186–210, 2026. https://doi.org/10.1016/j.laa.2025.09.001
\bibitem{ref-10} Chi Zhang, \textit{An Attack on $p$-adic Lattice Public-key Cryptosystems and Signature Schemes}, Designs, Codes and Cryptography, \textbf{93}(7), 2695–2716, 2025. https://doi.org/10.1007/s10623-025-01618-8
\bibitem{ref-11} Chi Zhang and Yingpu Deng, \textit{An Explicit Construction of Orthogonal Basis in $p$-adic Fields}, Communications in Mathematical Research, 2025. https://doi.org/10.4208/cmr.2025-0025
\end{thebibliography}
\end{document}